\documentclass[11pt]{article}
\usepackage{amsmath,amsfonts,amsthm, amssymb}
\usepackage{enumerate}

\newtheorem{defn}{{\bf Definition}}[section]
 \newtheorem{lemma}[defn]{{\bf
Lemma}} \newtheorem{prop}[defn]{{\bf Proposition}}
\newtheorem{theo}[defn]{{\bf Theorem}} \newtheorem{cor}[defn]{{\bf
Corollary}} \newtheorem{remark}[defn]{{\bf Remark}}

\font\bbb=msbm10 scaled\magstep1

\newcommand{\FF}{\mbox{\bbb F}}

\newcommand{\ZZ}{\mbox{\bbb Z}} 

\font\bbb=msbm8 scaled\magstep1

\newcommand{\TPSSD}{S^{\hspace{.2mm}d-1} \mbox{$\times
\hspace{-2.8mm}_{-}$} \, S^{\hspace{.1mm}1}}

 \newcommand{\TPPSS}{\kern.24em \rule width.08em height1.5ex
depth-.08ex \kern-.36em \times}

\newcommand{\lk}[2]{{\rm lk}_{#1}(#2)} 
\newcommand{\st}[2]{{\rm st}_{#1}(#2)}
\newcommand{\skel}[2]{{\rm skel}_{#1}{(#2)}}

\newcommand{\Kd}{{\mathcal{K}}(d)}
\newcommand{\lKd}{\overline{\mathcal{K}}(d)}

\begin{document}
\title{\bf Non-existence of tight neighborly triangulated manifolds with
\boldmath{$\beta_1=2$}}
\author{{\bf Nitin Singh}\\
\small  Department of Mathematics, Indian Institute of
Science, Bangalore 560\,012, India. \\
\small Email: nitin@math.iisc.ernet.in}

\date{To appear in `{\bf Advances in Geometry}'}

\maketitle
\vspace{-5mm}

\noindent{\bf Abstract:\ } {\small All triangulated $d$-manifolds
satisfy the inequality $\binom{f_0-d-1}{2} \geq \binom{d+2}{2}\beta_1$
for $d\geq 3$. A triangulated $d$-manifold is called \emph{tight
neighborly} if it attains equality in the bound. For each $d\geq 3$, a
$(2d+3)$-vertex tight neighborly triangulation of the
$S^{\,d-1}$-bundle over $S^1$ with $\beta_1=1$ was constructed by
K\"{u}hnel in 1986. In this paper, it is shown that there does not
exist a tight neighborly triangulated manifold with $\beta_1=2$.  In
other words, there is no tight neighborly triangulation of
$(S^{\,d-1}\!\times S^1)^{\#2}$ or $(\TPSSD)^{\#2}$ for $d\geq 3$. A
short proof of the uniqueness of K\"{u}hnel's complexes for $d\geq 4$,
under the assumption $\beta_1\neq 0$ is also presented.  }

\bigskip

\noindent {\small {\em MSC 2000\,:} 57Q15, 57R05.

\noindent {\em Keywords\,:} Stacked sphere; Triangulated
manifolds; Tight neighborly triangulation.  }

%%% SECTION - INTRODUCTION %%%
\section*{Introduction}

Tight neighborly triangulations were introduced by Lutz, Sulanke and
Swartz in \cite{LSS}. Using a result of Novik and Swartz \cite{NS},
the authors in \cite{LSS} obtained a lower bound on the minimum number
of vertices in a triangulation of a $d$-manifold in terms of its
$\beta_1$ coefficient (see Proposition \ref{prop:lss}).  Triangulations
that meet the lower bound are called {\em tight neighborly}. Thus
tight neighborly triangulations are vertex minimal triangulations.
Effenberger \cite{Effenberger} showed that for $d\geq 4$, tight
neighborly triangulated manifolds are $\ZZ_2$-tight. In conjunction
with a recent result of Bagchi and Datta \cite{BDTIGHT}, this implies
they are {\em strongly minimal} for $d\geq 4$. Apart from the
following classes of vertex-minimal triangulations, namely
\begin{itemize} 
\item the $(d+2)$-vertex triangulation of the $d$-sphere $S^d$, 
\item K\"{u}hnel's $(2d+3)$-vertex triangulations \cite{K86} of
$S^{\,d-1}\!\times S^1$ (for even $d$) and $\TPSSD$ (for odd $d$),
\end{itemize} 
very few examples of tight neighborly triangulations are known. A
first sporadic example, a $15$-vertex triangulation of a $4$-manifold
with $\beta_1=3$, was obtained by Bagchi and Datta \cite{BDW}.
Recently in \cite{BDNS}, we obtained tight neighborly triangulations
of $4$-manifolds with $21,26$ and $41$ vertices.  For $\beta_1=2$, the
parameters (integer solutions of the tight neighborliness condition) for
the first few possible tight neighborly triangulations are
$(f_0,d)=(35,13)$ and $(f_0,d)=(204,83)$. The main result of this
paper shows that such triangulations do not exist. In this article,
unless the field is explicitly stated, we assume
$\beta_1(X)=\beta_1(X;\ZZ_2)$.

%%% SECTION - PRELIMINARIES %%%
\section{Preliminaries} 
All graphs considered here are simple (i.e., undirected with no loops
or multiple edges). For the standard terminology on graphs, see
\cite[Chapter 1]{Diestel} for instance. For a graph $G$, $V(G)$ and
$E(G)$ will denote its set of vertices and edges respectively. A graph
$G$ is said to be $k$-\emph{connected} if $|V(G)|\geq k+1$ and $G-U$
is connected for all $U\subseteq V(G)$ with $|U|<k$. It is easily seen
that for a $k$-connected graph $G$, $d_G(v)\geq k$ for all $v\in
V(G)$.

All simplicial complexes considered here are finite and abstract.
By a triangulated manifold/sphere/ball, we mean an abstract
simplicial complex whose geometric carrier is a topological
manifold/sphere/ball. We identify two complexes if they are
isomorphic. A $d$-dimensional simplicial complex is called {\em
pure} if all its maximal faces (called {\em facets}) are
$d$-dimensional. A $d$-dimensional pure simplicial complex is said
to be a {\em weak pseudomanifold} if each of its $(d - 1)$-faces
is in at most two facets. For a $d$-dimensional weak
pseudomanifold $X$, the {\em boundary} $\partial X$ of $X$ is the
pure subcomplex of $X$ whose facets are those $(d-1)$-dimensional
faces of $X$ which are contained in unique facets of $X$. The {\em
dual graph} $\Lambda(X)$ of a pure simplicial complex $X$ is the
graph whose vertices are the facets of $X$, where two facets are
adjacent in $\Lambda(X)$ if they intersect in a face of
codimension one. A {\em pseudomanifold} is a weak pseudomanifold
with a connected dual graph. All connected triangulated manifolds
are automatically pseudomanifolds.

If $X$ and $Y$ are simplicial complexes with disjoint vertex sets, we
define $X\ast Y$ to be the simplicial complex whose faces are the
(disjoint) unions of faces of $X$ with faces of $Y$. When $X$ consists
of a single vertex $x$, we write $x\ast Y$ for $X\ast Y$. For a face
$\alpha$ of $X$, the {\em link} of $\alpha$ in $X$, denoted by
$\lk{X}{\alpha}$ is the subcomplex of $X$ consisting of all faces $\beta$
such that $\alpha\cap \beta=\emptyset$ and $\alpha\cup \beta$ is a face of
$X$.
When $\alpha$ consists of a single vertex $v$, we write $\lk{X}{v}$ instead
of $\lk{X}{\{v\}}$. For a vertex $v$ of $X$, we define the {\em star} of
$v$ in $X$, denoted by $\st{X}{v}$ as the cone $v\ast \lk{X}{v}$. The
subcomplex of $X$ consisting of faces of dimension at most $k$ is called
the {\em $k$-skeleton} of $X$, and is denoted by $\skel{k}{X}$. By the {\em
edge graph} of a simplicial complex $X$, we mean its $1$-skeleton.

If $X$ is a $d$-dimensional simplicial complex then, for $0\leq j
\leq d$, the number of its $j$-faces is denoted by $f_j = f_j(X)$.
The vector $(f_0, \dots, f_d)$ is called the {\em face vector} of
$X$ and the number $\chi(X) := \sum_{i=0}^{d} (-1)^i f_i$ is
called the {\em Euler characteristic} of $X$. As is well known,
$\chi(X)$ is a topological invariant, i.e., it depends only on the
homeomorphic type of $|X|$. A simplicial complex $X$ is said to be
{\em $l$-neighbourly} if any $l$ vertices of $X$ form a face of
$X$. In this paper, by a {\em neighborly} complex, we shall mean a
$2$-neighborly complex.

A {\em standard $d$-ball} is a pure $d$-dimensional simplicial
complex with one facet. The standard ball with facet $\sigma$ is
denoted by $\overline{\sigma}$. A $d$-dimensional pure simplicial
complex $X$ is called a {\em stacked $d$-ball} if there exists a
sequence $B_1, \dots, B_m$ of pure simplicial complexes such that
$B_1$ is a standard $d$-ball, $B_m=X$ and, for $2\leq i\leq m$,
$B_i = B_{i-1}\cup \overline{\sigma_i}$ and $B_{i-1} \cap
\overline{\sigma_i} = \overline{\tau_i}$, where $\sigma_i$ is a
$d$-face and $\tau_i$ is a $(d-1)$-face of $\sigma_i$. Clearly, a
stacked ball is a pseudomanifold. A simplicial complex is called a
{\em stacked $d$-sphere} if it is the boundary of a stacked
$(d+1)$-ball.
A trivial induction on $m$ shows that a stacked $d$-ball actually
triangulates a topological $d$-ball, and hence a stacked
$d$-sphere triangulates a topological $d$-sphere. If $X$ is a
stacked ball then clearly $\Lambda(X)$ is a tree. 

\begin{prop}[Datta and Singh \cite{BDNS}]\label{prop:bdns}
Let $X$ be a pure $d$-dimensional simplicial complex.
\begin{enumerate}[{\rm (a)}]
\item If $\Lambda(X)$ is a tree then $f_0(X) \leq f_d(X) +d$.
\item $\Lambda(X)$ is a tree and $f_0(X) = f_d(X) +d$ if and only
if $X$ is a stacked ball.
\end{enumerate}
\end{prop}

In \cite{Walkup}, Walkup defined a class $\Kd$ of triangulated
$d$-manifolds where all vertex links are stacked $(d-1)$-spheres.
Analogously, let us define the class $\lKd$ of triangulated
$d$-manifolds all whose vertex links are stacked $(d-1)$-balls.
Clearly the class $\lKd$ consists of manifolds with boundary. In
\cite{bd16}, Bagchi and Datta have shown that for $d\geq 4$, the
members of $\Kd$ can be obtained as boundary of members of
$\overline{\mathcal K}(d+1)$. For a simplicial complex $X$, let
$V(X)$ denote its vertex set. For a set $S$, let $\binom {S}{\leq k}$
denote the collection of subsets of $S$ of cardinality at most $k$.
From the results in \cite{bd16}, we have\,:

\begin{prop}[Bagchi and Datta]\label{prop:bagchidatta} Let
$d\geq 4$ and $M\in{\mathcal{K}}(d)$. Then $\overline{M}$ defined by
\begin{equation}\label{eq:bagchidatta}
\overline{M} :=
\left\{\alpha\subseteq V(M) : \binom{\alpha}{\leq 3}\subseteq M \right\}
\end{equation}
is the unique member of $\overline{\mathcal K}(d+1)$ such that
$\partial \overline{M}=M$.  
\end{prop}

In the above construction, notice that $V(\overline{M})=V(M)$ and
$\overline{M}$ is neighborly if and only if $M$ is neighborly.

Let $\sigma_1,\sigma_2$ be two facets of a pure simplicial complex
$X$. Let $\psi: \sigma_1\rightarrow \sigma_2$ be a bijection such that
$x$ and $\psi(x)$ have no common neighbor in the edge graph ($1$-skeleton)
of $X$ for each $x\in \sigma_1$. Let $X^\psi$ denote the complex
obtained by identifying $x$ with $\psi(x)$ in $X\backslash
\{\sigma_1,\sigma_2\}$. Then $X^\psi$ is said to be obtained from $X$
by a {\em combinatorial handle addition}. We know the following:

\begin{prop}[Kalai \cite{ka}]\label{prop:kalai}
For $d\geq 4$, a connected simplicial complex $X$ is in ${\cal
K}(d)$ if and only if $X$ is obtained from a stacked $d$-sphere by
$\beta_1(X)$ combinatorial handle additions. In consequence, any
such $X$ triangulates either $(S^{\,d -1}\!\times S^1)^{\#
\beta_1}$  or $(\TPSSD)^{\#\beta_1}$ depending on whether $X$ is
orientable or not. $($Here $\beta_1 = \beta_1(X) =
\beta_1(X;\ZZ_2).)$
\end{prop}

In the above, $S^{\,d-1}\!\times S^1$ denotes the (orientable) sphere
product with a circle, while $\TPSSD$ denotes the (non-orientable)
twisted sphere product with a circle. As usual $X^{\#k}$ denotes the
connected sum of $k$ copies of the manifold $X$. 
From Proposition \ref{prop:kalai} and Theorem $4$ in \cite{LSS}, we
infer the following\,:

\begin{prop}\label{prop:lss} Let $X$ be a connected triangulated
$d$-manifold. Then $X$ satisfies
\begin{equation}\label{eq:tn}
\binom{f_0-d-1}{2} \geq \binom{d+2}{2}\beta_1.
\end{equation}
Moreover for $d\geq 4$, the equality holds if and only if $X$ is a
neighborly member of ${\cal K}(d)$.
\end{prop}

For $d\geq 3$, a triangulated $d$-manifold is called {\em tight
neighborly} if it satisfies (\ref{eq:tn}) with equality.

For a field $\mathbb{F}$, a $d$-dimensional simplicial complex $X$
is called {\em tight with respect to $\mathbb{F}$} (or {\em
$\mathbb{F}$-tight}) if {\rm (i)} $X$ is connected, and {\rm (ii)}
for all induced sub-complexes $Y$ of $X$ and for all $0\leq j\leq
d$, the morphism $H_j(Y;\mathbb{F})\rightarrow H_j(X;\mathbb{F})$
induced by the inclusion map $Y\hookrightarrow X$ is injective. 

A $d$-dimensional simplicial complex $X$ is called {\em strongly minimal} if
$f_i(X)\leq f_i(Y)$, $0\leq i\leq d$, for every triangulation $Y$ of the
geometric carrier $|X|$ of $X$.

Effenberger \cite{Effenberger} proved that for $d\geq 4$, tight
neighborly triangulated $d$-manifolds are $\ZZ_2$-tight. Bagchi and Datta
\cite{BDTIGHT} proved that for $d\geq 4$, $\FF$-tight members of
Walkup's class $\Kd$ are strongly minimal. Thus for $d\geq 4$,
tight neighborly triangulated $d$-manifolds are strongly minimal.

%%% SECTION - MAIN RESULT %%%
\section{Non-existence of tight neighborly manifolds with
$\beta_1=2$}\label{sec:mainresult}

The following is the main result of this article.

\begin{theo}\label{thm:inequality} For $d\geq 4$, there does not exist
a tight neighborly triangulated $d$-manifold $M$ with $\beta_1(M)=2$.
\end{theo}
\begin{proof}[{\bf Proof-Sketch:}]
Suppose there exists a tight
neighborly $d$-manifold $M$ with $\beta_1(M)=2$. Then by
Proposition \ref{prop:lss}, $M$ is a neighborly member of $\Kd$. By Proposition
\ref{prop:bagchidatta}, there exists a neighborly member
$\overline{M}$ of $\overline{\cal
K}(d+1)$ such that $\partial \overline{M}=M$ and
$V(\overline{M})=V(M)$. The proof rests on the following observation
which follows from Corollary \ref{cor:degthreemore}:
\smallskip
\par\noindent{\it Let $T$ be the set of facets of $\overline{M}$ with
degree three or more in $\Lambda(\overline{M})$. Then the facets in
$T$ together contain all the vertices of $\overline{M}$.} 
\smallskip

\noindent Since $\overline{M}$ is $(d+1)$-dimensional, we have the obvious
inequality $f_0(M) = f_0(\overline{M})\leq |T|(d+2)$. Since $M$ is tight
neighborly, we have $(f_0(M)-d-1)(f_0(M)-d-2)=\beta_1(M)(d+1)(d+2)$.
For the case $\beta_1(M)=2$, we shall see that the inequality and the
equation cannot be simultaneously satisfied, thus proving the theorem.
\end{proof}

The following are used in the proof of Theorem \ref{thm:inequality}. In the
results below we shall assume that $M\in \overline{\cal K}(d)$ is not the
standard $d$-ball $B^{d}_{d+1}$.

\begin{lemma}\label{lem:twoconnected} Let $M\in\lKd$. If $M$ is
neighborly, then the dual graph of $M$ is $2$-connected.
\end{lemma}

\begin{proof}
Let $\Lambda$ denote the dual graph of $M$. Suppose $\Lambda$ is not
$2$-connected. Then there exists $\sigma\in V(\Lambda)$ such that
$\Lambda - \sigma$ is disconnected. Let $C_1, C_2$ be different
components of $\Lambda - \sigma$. Choose $\sigma_1\in C_1$ and
$\sigma_2 \in C_2$. Now choose $x_1\in  \sigma_1$ and $x_2\in
\sigma_2$ such that $x_1,x_2\not\in \sigma$. Note that $\sigma,
\sigma_1$ and $\sigma_2$ are facets of $M$, while $x_1$ and $x_2$ are
vertices of $M$.  Let $V_1$ denote the
set of vertices of $\Lambda$ that contain $x_1$ and $V_2$ denote
the set of vertices of $\Lambda$ containing $x_2$. Observe that
the subgraph of $\Lambda$ induced by $V_1$ is precisely the dual
graph of $\st{M}{x_1}$, i.e., $\Lambda(\lk{M}{x_1})\cong
\Lambda(\st{M}{x_1}) = \Lambda[V_1]$. Similarly
$\Lambda(\lk{M}{x_2})\cong \Lambda[V_2]$. Since $M\in \lKd$, the
vertex-links are stacked balls, and hence, by Proposition
\ref{prop:bdns}, their dual graphs are
trees. Thus $V_1$ and $V_2$ induce trees on $\Lambda$.  Since
$x_1,x_2\not\in \sigma$ we conclude that $V_1$ and $V_2$ induce
trees on $\Lambda - \sigma$. Now since $\sigma_1\in V_1$ and
$\sigma_2\in V_2$ we conclude that $\Lambda[V_1]\subseteq C_1$ and
$\Lambda[V_2]\subseteq C_2$.  Therefore $V_1\cap V_2=\emptyset$.
Therefore $x_1x_2$ is not a simplex in $M$, contradicting the
neighborliness of $M$. This proves the lemma.
\end{proof}

\begin{lemma}\label{lem:numfacets} Let $M\in \overline{\mathcal
K}(d)$ be neighborly. Then each $x\in V(M)$ is contained in
$f_0(M)-d$ facets of $M$, which induce a tree on $\Lambda(M)$.
\end{lemma}

\begin{proof}
Let $L =\lk{M}{x}$ denote the link of $x$ in $M$. Then $L$ is a
stacked $(d-1)$-ball. Since $M$ is neighborly, we have
$f_0(L)=f_0(M)-1$. Then from Proposition \ref{prop:bdns}, we see that
$f_{d-1}(L)=f_0(M)-d$ and that $\Lambda(L)$ is a tree. Finally we
observe that $\Lambda(L)\cong \Lambda(M)[V_x]$, where $V_x$ is the
set of facets of $\Lambda(M)$ containing $x$. This completes the
proof.
\end{proof}

\begin{lemma}\label{lem:dgone} Let $M\in \overline{\mathcal
K}(d)$ be neighborly. Then $\Lambda(M)$ has $n(n-d)/(d+1)$
vertices and $n(n-d-1)/d$ edges, where $n=f_0(M)$.
\end{lemma}

\begin{proof}
Let $\Lambda(M)$ have $\nu$ vertices and $\varepsilon$ edges. We
count the pairs $(x,\tau)$ where $x\in V(M)$, $\tau$ is a facet of
$M$ containing $x$. Now $x\in \tau\Rightarrow \tau\backslash
\{x\}\in \lk{M}{x}$. By Lemma \ref{lem:numfacets}, we conclude
that each $x$ appears in $n-d$ facets of $M$. Thus the number of
pairs is $n(n-d)$. However since each facet of $M$ is $d$-simplex,
each facet occurs with $d+1$ values of $x$. Counting this way
gives us $\nu(d+1)$ pairs. Equating the two we have
$\nu=n(n-d)/(d+1)$. To get the number of edges in the dual graph
we count the pairs
\begin{equation}\label{eq:edgepair}
(x, \sigma\tau) \mbox{ where } x\in V(M),
\sigma\tau\in E(\Lambda(M)) \mbox{ and } x\in \sigma\cap\tau.
\end{equation}
Let $S_x$ denote the star of $x$ in $M$. Then as seen previously
$\Lambda(S_x)$ is an $(n-d)$-vertex induced tree of $\Lambda(M)$.
Now note that $(\ref{eq:edgepair})$ is equivalent to saying that
$\sigma,\tau$ are facets in $S_x$, and moreover they form an edge
in $\Lambda(S_x)$. Since $\Lambda(S_x)$ has $n-d-1$ edges, we see
that each $x$ contributes $n-d-1$ pairs.  Thus the number of pairs
is $n(n-d-1)$. However we can count differently.  Consider the
pair of facets $\{\sigma,\tau\}$ forming an edge in $\Lambda(M)$.
Then $|\sigma\cap\tau|=d$ and hence each edge occurs with $d$
values of $x$.  This gives the number of pairs as $d\varepsilon$.
Equating the two values we get $\varepsilon=n(n-d-1)/d$.
\end{proof}

\begin{cor}\label{cor:cycle} Let $M\in \lKd$ be neighborly.
Then $f_0(M)\geq 2d+1$ and the equality holds if and only if
$\Lambda(M)$ is a cycle.
\end{cor}

\begin{proof} Let $\nu$ and $\varepsilon$ denote the number of
vertices and edges of $\Lambda(M)$ respectively. Let ${\cal V} :=
V(\Lambda(M))$ denote the vertex set of the graph $\Lambda(M)$. By Lemma
\ref{lem:twoconnected}, $\Lambda(M)$ is $2$-connected. Thus all vertex
degrees are at least two, i.e., $d_{\Lambda(M)}(\sigma)\geq 2$ for all
$\sigma\in {\cal V}$.  Let $T=\{\sigma\in {\cal V}: d_{\Lambda(M)}(\sigma)
\geq 3\}$.  Then $\Lambda(M)$ is a cycle if and only if $T=\emptyset$. Now
we have, $2\nu \leq \sum_{\sigma\in {\cal
V}}d_{\Lambda(M)}(\sigma)=2\varepsilon$,
or $\epsilon\geq \nu$. Clearly the equality occurs
when $T=\emptyset$. Using Lemma \ref{lem:dgone} we have,
\begin{equation*}
\frac{n(n-d-1)}{d} \geq \frac{n(n-d)}{d+1}, \text{ where } n=f_0(M)
\end{equation*}
Thus $f_0(M)=n\geq 2d+1$ and equality occurs only when
$T=\emptyset$, or equivalently when $\Lambda(M)$ is a cycle.
\end{proof}

\begin{lemma}\label{lem:beta1} For $d\geq 4$, let $M\in \Kd$ be
neighborly and let $\overline{M}\in \overline{\cal K}(d+1)$ be such
that $\partial \overline{M}=M$. If $\nu$ and $\varepsilon$ denote the
number of vertices and edges of $\Lambda(\overline{M})$
respectively, then $\beta_1(M)=\varepsilon-\nu+1$.
\end{lemma}

\begin{proof}
Since $M\in \Kd$ is neighborly, by Proposition \ref{prop:lss}, we have
$\beta_1(M)=\binom{n-d-1}{2}/\binom{d+2}{2}$. Since $\overline{M}\in
\overline{\mathcal K}(d+1)$, by Lemma \ref{lem:dgone} we have
$\nu=n(n-d-1)/(d+2)$, $\varepsilon=n(n-d-2)/(d+1)$, where
$n=f_0(M)=f_0(\overline{M})$.  Then it follows that
$\beta_1(M)=\varepsilon-\nu+1$.
\end{proof}

\begin{defn} {\rm Let $M$ be a neighborly member of
$\overline{\mathcal K}(d)$. A set $S\subseteq V(\Lambda(M))$ is
said to be \emph{critical} in $M$ if each of the connected components of
$\Lambda(M)-S$ contains fewer than $f_0(M)-d$ vertices. A set of
facets is said to be a \emph{cover} of $M$ if they together
contain all the vertices.  }
\end{defn}

Observe the following.

\begin{lemma}\label{lem:criticaliscover}
Let $M\in \overline{\mathcal K}(d)$ be neighborly. If $S\subseteq
V(\Lambda(M))$ is critical in $M$, then $S$ is a cover of $M$.
\end{lemma}

\begin{proof}
Since $S$ is critical in $M$, each component of $\Lambda(M)-S$ is of size at
most $f_0(M)-d-1$. Let $x$ be an arbitrary vertex of $M$. Let $V_x$ be
the set of facets of $M$ containing $x$. By Lemma \ref{lem:numfacets},
we know that $\Lambda(M)[V_x]$ is an induced tree with $f_0(M)-d$
vertices. Hence $V_x$ must intersect $S$, or equivalently a facet in
$S$ contains $x$. Since $x$ was arbitrary, we conclude that the facets
in $S$ contain all the vertices, and hence $S$ is a cover of $M$. 
\end{proof}

\begin{lemma}\label{lem:pathlemma} Let $M\in \overline{\mathcal
K}(d)$ be neighborly with $f_0(M)>2d+1$. Let $u_0u_1\ldots u_r$ be
a path in $\Lambda(M)$. Let $x_i$ be the unique element of
$u_{i-1}\backslash u_i$ for $1\leq i\leq r$. If all the internal
vertices of the path have degree at most two in $\Lambda(M)$, then
we have the following\,:
\begin{enumerate}[{\rm (a)}]
\item $x_i\neq x_j$ for $i\neq j$.
\item $x_i\in u_0$ for all $1\leq i\leq r$.
\item $r\leq d+1$.
\end{enumerate}
\end{lemma}

\begin{proof} We first prove (a). If possible, let there exist $i,j$ with
$i<j$ such that $x_i=x_j=x$. Then by definition $x\in
u_{i-1},u_{j-1}$ but $x\not\in u_i,u_j$; hence $j>i+1$. Since the
set of facets containing $x$ must induce a tree, we conclude that
there is a $u_{i-1}$-$u_{j-1}$ path in $\Lambda(M)-\{u_i,u_j\}$.
However we see that if all the internal vertices have degree at
most two, this is not possible. This proves (a).
\par\noindent{Suppose} (b) is not true. Let $i$ be the minimum
such that $x_i\not\in u_0$. As $x_1\in u_0$, we have $i>1$. By
minimality of $i$, we must have $\{x_1,\ldots,x_{i-1}\}\subseteq
u_0$. Since $|u_0|=d+1$, we see that $i\leq d+2$. But then we have
$x_i\in u_{i-1}$, $x_i\not\in u_i$ and $x_i\not\in u_0$. Let $V_i$
denote the set of facets containing $x_i$. Since $\Lambda(M)[V_i]$
is a tree, it must be contained in a unique component of
$\Lambda(M)- \{u_0,u_i\}$. Since $x_i \in u_{i-1}$,
$\Lambda(M)[V_i]$ is contained in the component of
$\Lambda(M)-\{u_0,u_i\}$ containing $u_{i-1}$, which according to
our assumptions is the path $u_1\ldots u_{i-1}$. Thus
$V_i\subseteq \{u_1,\ldots, u_{i-1}\}$. But then $|V_i|\leq
d+1<f_0(M)-d$, contradicting Lemma \ref{lem:numfacets}. This
proves (b). Since $|u_0|=d+1$, it is readily seen that (a) and (b)
yield (c). This completes the proof of the lemma.
\end{proof}

\begin{cor}\label{cor:degthreemore} For $d\geq 4$, let $M\in
\overline{\mathcal K}(d)$ be neighborly. Let $T=\{\sigma \in
V(\Lambda(M)) : d_{\Lambda(M)}(\sigma)\geq 3\}$. If $f_0(M)>2d+1$,
then $T$ is a cover of $M$.
\end{cor}

\begin{proof} From Corollary \ref{cor:cycle},  $f_0(M)>2d+1$ implies
$T\neq \emptyset$. Clearly each component of $\Lambda(M)-T$ is an
induced path in $\Lambda(M)$. Then by Lemma \ref{lem:pathlemma},
we see that each component of $\Lambda(M)-T$ is of size at most
$d+1<f_0(M)-d$.  Hence $T$ is critical in $M$. Therefore by Lemma
\ref{lem:criticaliscover}, we conclude that $T$ is a cover of $M$.
\end{proof}

We are now in a position to give a complete proof of Theorem
\ref{thm:inequality}.

\begin{proof}[{\bf Proof (Theorem \ref{thm:inequality}):}]
If possible let $M$ be a tight neighborly $d$-manifold for $d\geq 4$
with $\beta_1=\beta_1(M)=2$. Let $n=f_0(M)$. Then we have
$\binom{n-d-1}{2}=2\binom{d+2}{2}$, or equivalently,
\begin{equation}\label{eq:tightneighborly}
(n-d-1)(n-d-2) = 2(d+2)(d+1)
\end{equation}
By Proposition \ref{prop:lss}, $M$ must be a neighborly member of $\Kd$.
Then by Proposition \ref{prop:bagchidatta}, there exists
$\overline{M}\in \overline{\cal K}(d+1)$ such that $\partial
\overline{M}=M$. Further we know that $\overline{M}$ is neighborly and
$V(\overline{M})=V(M)$. By Corollary \ref{cor:cycle}, we have
$f_0(M) = n = f_0(\overline{M}) \geq 2(d+1)+1 = 2d+3$. For $n=2d+3$,
we see that $\Lambda(\overline{M})$ is a cycle, and hence by Lemma
\ref{lem:beta1}, $\beta_1(M)=1$. Thus we can assume $n>2d+3$. Let
${\cal V}$ and ${\cal E}$ denote the vertex and edge set of
$\Lambda(\overline{M})$ respectively. Let $T\subseteq {\cal V}$ be the
set of facets of $\overline{M}$ with degree three or more in
$\Lambda(\overline{M})$. By Corollary \ref{cor:degthreemore}, $T$ is a
cover of $\overline{M}$. Since $\overline{M}$ is $(d+1)$-dimensional,
we have $n\leq |T|(d+2)$. We now estimate $|T|$. By Lemma
\ref{lem:twoconnected}, $\Lambda(\overline{M})$ is $2$-connected. Hence
$d_{\Lambda(\overline{M})}(\sigma)\geq 2$ for all $\sigma\in {\cal
V}$. Now we have,
\begin{equation*}
|T|\leq \sum_{\sigma\in {\cal V}} (d_{\Lambda(\overline{M})}(\sigma)-2) =
	\sum_{\sigma\in {\cal V}} d_{\Lambda(\overline{M})}(\sigma) -
2|{\cal V}| = 2(|{\cal E}|-|{\cal V}|).
\end{equation*}
By Lemma \ref{lem:beta1}, we have $|{\cal E}|-|{\cal V}|=\beta_1-1=1$.
Thus $|T|\leq 2$, and hence $n\leq 2(d+2)=2d+4$. But then,
\begin{equation*}
(n-d-1)(n-d-2)\leq (d+2)(d+3) < 2(d+2)(d+1)
\end{equation*}
which contradicts (\ref{eq:tightneighborly}). This completes the proof.
\end{proof}

\begin{remark} {\rm Theorem \ref{thm:inequality} shows that there does not exist
a tight neighborly triangulation with $(\beta_1,d,f_0)=(2,13,35)$,
which was one of the open cases in \cite[Section 4]{Effenberger}.
The next few triples $(\beta_1,d,f_0)$ with $\beta_1=2$ are
$(2,83,204)$ and $(2,491,1189)$, which also do not exist by
Theorem \ref{thm:inequality}. Indeed in conjunction with
Propositions \ref{prop:kalai} and \ref{prop:lss}, we get\,:}
\end{remark}

\begin{cor}
If $X$ is an $n$-vertex triangulation of $(S^{\,d-1}\!\times S^1)^{\#2}$ or
$(\TPSSD)^{\#2}$ and $d\geq 4$, then
\begin{equation*}
\binom{n-d-1}{2}\geq d^2+3d+3.
\end{equation*}
\end{cor}

%%% SECTION - UNIQUENESS OF KUEHNEL's TORUS %%%

\section{Uniqueness of K\"{u}hnel's Tori} 
For $d\geq 2$, $d$-dimensional K\"{u}hnel's torus $K^d_{2d+3}$ \cite{K86}
is defined as the boundary of the $(d+1)$-dimensional pseudomanifold
$\overline{K}^{\,d+1}_{2d+3}$ on the vertex set $\{0,\ldots,2d+2\}$ with
facets $\{\{ i,i+1,\ldots,i+d+1 \} : 0\leq i\leq 2d+2\}$, where the
addition is modulo $2d+3$. For even $d$, $K^d_{2d+3}$ triangulates the
sphere product $S^{\,d-1}\!\times S^1$ and for odd $d$, it triangulates the
twisted product $\TPSSD$. The following result was proved in \cite{BDSB}.

\begin{prop}
For $d\geq 3$, K\"{u}hnel's torus $K^d_{2d+3}$ is the only
non-simply connected $(2d+3)$-vertex triangulated manifold of
dimension $d$.
\end{prop}

For $d\geq 4$, the uniqueness of $K^d_{2d+3}$ was also proved in
\cite{CSE} for the bigger class of homology $d$-manifolds, but
with assumption $\beta_1\neq 0$ and $\beta_2=0$. We prove the
above result for $d\geq 4$, under the assumption $\beta_1\neq 0$.
More specifically we prove\,:

\begin{theo}
For $d\geq 4$, let $M$ be a triangulated $d$-manifold with $2d+3$
vertices and $\beta_1(M;\ZZ_2)\neq 0$. Then $M\cong K^d_{2d+3}$.
\end{theo}

\begin{proof}
By Proposition \ref{prop:lss}, we must have $\beta_1(M)=1$ and $M$ must be
tight neighborly. Therefore $M\in \Kd$. By Proposition
\ref{prop:bagchidatta}, there exists $\overline{M}\in
\overline{\cal K}(d+1)$ such that $\partial \overline{M}=M$. Let
$V(M)=V(\overline{M})=\{0,1,\ldots,2d+2\}$. Since
$f_0(\overline{M})=2d+3=2(d+1)+1$, by Corollary \ref{cor:cycle},
we conclude that $\Lambda(\overline{M})$ is a cycle with $2d+3$
vertices. Let $\sigma_0,\ldots,\sigma_{2d+2}$ be facets of
$\overline{M}$ such that $\Lambda(\overline{M})$ is the cycle
$\sigma_0\sigma_1\ldots\sigma_{2d+2}\sigma_0$. For $0\leq i\leq
2d+2$, let $V_i$ denote the set of facets of $\overline{M}$
containing the vertex $i$. By Lemma \ref{lem:numfacets}, $V_i$
induces a $(2d+3)-(d+1)=d+2$ vertex tree on
$\Lambda(\overline{M})$. In other words, $V_i$ induces a
$(d+1)$-length path on $\Lambda(\overline{M})$. There are exactly $2d+3$
induced paths of $\Lambda(\overline{M})$ with length $d+1$, namely the
paths $P_k=\sigma_{k-d-1}\sigma_{k-d}\cdots\sigma_k$ for $0\leq k\leq
2d+2$. Next we show that the sets $V_i$ are distinct for $i=0,\ldots,
2d+2$. If not, suppose $V_i=V_j$ for some $i\neq j$.  Let
$P=\Lambda(\overline{M})[V_i]=\Lambda(\overline{M})[V_j]$.  Clearly, $P$ is
an induced path. Let $\sigma$ be an end-vertex of $P$ and let $\tau$ be a
neighbor of $\sigma$ in $\Lambda(\overline{M})$, not on $P$ (such a
neighbor exists, as the degree of each vertex in $\Lambda(\overline{M})$ is
at least two). But then $\{i,j\}\subseteq \sigma\backslash \tau$, which is
not possible, as $\sigma\tau$ is an edge of $\Lambda(\overline{M})$. Thus
$V_i\neq V_j$ for $i\neq j$.  Therefore, for each $0\leq k\leq 2d+2$, there
exists a unique $l$ such that $\Lambda(\overline{M})[V_l]=P_k$. Let us
denote this association by $\varphi$, i.e., $\varphi(k)=l$ if
$\Lambda(\overline{M})[V_l] = P_k$. Then we have $\sigma_i=\{
\varphi(i),\varphi(i+1),\ldots,\varphi(i+d+1)\}$. We notice that in this
case $\varphi$ is a simplicial isomorphism from
$\overline{K}^{\,d+1}_{2d+3}$ to $\overline{M}$. Therefore $K^d_{2d+3} =
\partial \overline{K}^{\,d+1}_{2d+3} \cong \partial \overline{M} = M$.
This completes the proof.  
\end{proof}

\bigskip

\noindent{\bf Acknowledgement\,:} The author thanks the anonymous referee
for useful comments regarding the presentation of the paper. The author
also thanks Basudeb Datta and Bhaskar Bagchi for their valuable comments and
suggestions. The author would also like to thank `IISc Mathematics
Initiative' and `UGC Centre for Advanced Study' for support.

\end{document}